\documentclass[10pt]{article}

\usepackage{preamble}

\begin{document}

\newcommand\restr[2]{{% we make the whole thing an ordinary symbol
  \left.\kern-\nulldelimiterspace % automatically resize the bar with \right
  #1 % the function
  \vphantom{\big|} % pretend it's a little taller at normal size
  \right|_{#2} % this is the delimiter
  }}

\makeatletter
% we use \prefix@<level> only if it is defined
\renewcommand{\@seccntformat}[1]{%
  \ifcsname prefix@#1\endcsname
    \csname prefix@#1\endcsname
  \else
    \csname the#1\endcsname\quad
  \fi}
% define \prefix@section
%\newcommand\prefix@section{Section \thesection: }
\makeatother

\makeatletter
\newcommand{\colim@}[2]{%
  \vtop{\m@th\ialign{##\cr
    \hfil$#1\operator@font colim$\hfil\cr
    \noalign{\nointerlineskip\kern1.5\ex@}#2\cr
    \noalign{\nointerlineskip\kern-\ex@}\cr}}%
}
\newcommand{\colim}{%
  \mathop{\mathpalette\colim@{\rightarrowfill@\textstyle}}\nmlimits@
}
\makeatother

\newcommand\rightthreearrow{%
        \mathrel{\vcenter{\mathsurround0pt
                \ialign{##\crcr
                        \noalign{\nointerlineskip}$\rightarrow$\crcr
                        \noalign{\nointerlineskip}$\rightarrow$\crcr
                        \noalign{\nointerlineskip}$\rightarrow$\crcr
                }%
        }}%
}

\title{Tate-valued  Characteristic Classes}         % Enter your title between curly bracse
\author{Shachar Carmeli and Kiran Luecke}        % Enter your name between curly braces
\date{\today}          % Enter your date or \today between curly braces
\maketitle

\begin{abstract}
We define a projective variant of classical complex orientation theory. Using this, we construct a map of spectra which lifts the total Chern class, providing an alternative answer to an old question of Segal \cite{segal}, previously answered by Lawson et al \cite{lawsonetal}. We also lift and generalize the ``sharp'' construction of Ando-French-Ganter \cite{afg} to an operation on arbitrary $\EE_\infty$-complex orientations, thereby providing a rich source of new $\EE_\infty$-orientations for commutative ring spectra. In particular we give an $\EE_\infty$-lift of the Jacobi orientation, a generalization of the much-studied two variable elliptic genus. Finally, we construct some new complex orientations of periodic ring spectra as requested in \cite{hahnyuan}.
\end{abstract}

\tableofcontents

\section{Introduction}

\subsection{Background and main results}

Let $\omega: \MU\to R$ be a complex oriented ring spectrum. A common method for studying this complex orientation is as follows. First, one finds a ring map $f:R\to R'$, together with a simpler orientation $\omega':\MU\to R'$. Then $R'$ has two complex orientations, $f_*\omega $ and $\omega'$, and by comparing these one gains information about $\omega$. Perhaps the most commonly deployed instance of this produces the \emph{Hirzebruch series}: one takes $f$ to be the rationalization map $f:R\to R\otimes\Q$ and $\omega'$ the orientation $\omega_\text{add}:\MU\to \Q\to R\otimes \Q$ whose formal group law is the additive one. Each of these orientations of $R\otimes \Q$ produces an element in $R_*[[x]]\otimes \Q$, and their ratio is precisely the classical Hirzebruch series of $\omega$. By construction this series is an element in $(R_*[[x]]\otimes \Q)^\times$, so via the splitting principle it produces a stable exponential characteristic class, i.e. a map of $H$-spaces $\BU\to GL_1(R\otimes \Q)$, which we will call the \emph{Hirzebruch characteristic class}.

In this paper we will carry out a similar program, but replacing the map $R\to R\otimes \Q$ with the much less destructive Tate construction $R\to R^{t\T}$, where $\T\subset\C$ is the topological group of unit complex numbers. In particular, we will define a variant of complex orientation theory (cf. \Cref{defn_proj_or}) in which $R^{t\T}$ has a canonical orientation $\ortate$ which plays the role of $\omega_\text{add}$ in the method described above. We will call the resulting series (resp. exponential characteristic class) the \emph{Euler-Tate} series (resp. class).  Crucially, our construction will be carried out at the level of symmetric monoidal categories, and thus automatically produces $\EE_\infty$-\emph{lifts} of the characteristic classes in question, as opposed to mere maps of $H$-spaces. Our first main result is the following. 

\begin{introthm} (\ref{prop_char_series_tch}, \ref{cor:coherent_chern})
    Let $\omega: \MU\to R$ be an $\EE_\infty$-ring map, with associated formal group law $F$, and orientation parameter $t\in R^2\CP^\infty$. Note that $\pi_*R^{t\T}\simeq R_*((t))$. Then there is an associated ``Euler-Tate'' characteristic class 
    \[\tch:\bu\to \glone(R^{t\T})\] whose characteristic series is $t^{-1}(x+_F t)$. In particular, when $R=\Z$ the underlying map of spaces is the total Chern class with respect to $t^{-1}$.
\end{introthm}

This answers a question of Graeme Segal from 1975 \cite{segal}, which asks if such an $\EE_\infty$-lift of the total Chern class exists. The question was answered in the affirmative by Lawson et al \cite{lawsonetal}, using their theory of `algebraic cycles' to construct infinite loop spaces out of geometric objects.

Recall that $\MU\l n\r$ is defined to be the Thom spectrum of the $j$-homomorphism restricted to the $n$-th connective cover of $\ku$. For example, $\MU\l 0\r=\mathrm{MUP}$, $\MU\l2\r=\MU$, and $\MU\l 4\r=\MSU$. The question of when an $\MU\l n\r$ orientation, considered as a homotopy ring map $\MU\l n \r\to R$, admits an $\EE_\infty$-lift is a difficult problem of great interest (\cite{andocoherence}, \cite{ahr}, \cite{hahnyuan}, \cite{hoplaw}). The number of known examples of $\EE_\infty$ orientations is quite small. Our next result gives a big supply of them:  a \emph{lift} and \emph{generalization} of the Ando-French-Ganter construction of \cite{afg}. Their construction turns homotopy $\MU\l 2n\r$-orientations of an even periodic ring spectrum $E$ into homotopy $\MU\l2n-2\r$-orientations of $E^{t\T}$, when $n=0,1,2,3$. Our construction produces the following; we refer the reader to \Cref{section_sharp} for details.

\begin{introthm}(\ref{thm_sharp_is_sharp})
     Let $\omega: \MU\l 2n\r\to R$ be an $\EE_\infty$-ring map. Then there exists and $\EE_\infty$-ring map
     \[(\omega)^\#:\MU\l2n-2\r\to R^{t\T}\]
     such that when $n=0,1,2,3$ and $R$ is even periodic, this is an $\EE_\infty$-lift of the Ando-French-Gander construction.
    %Let $E$ be an even periodic $\EE_\infty$-ring spectrum with an $\EE_\infty$-orientation $\omega: \MU\l 2(p+1)\r\to E$ for $p\in \{0,1,2,3\}$. Then the $\EE_\infty$ orientation $(\omega)^\sharp \colon \MU\l 2p\r\to E^{t\T}$ from \Cref{constr_sharp} is an $\EE_\infty$ lift of the AFG-sharp orientation $(\omega)^{\mathrm{afg}\#}$ of \Cref{defn_afg_sharp}.
\end{introthm}

Our third main result is an immediate corollary, gotten by specializing to the case $R=\mathrm{tmf}$ with its $\EE_\infty$ ``string'' orientation. We obtain an $\EE_\infty$-lift of the Jacobi orientation of \cite{afg}, which is a generalization of the much-studied ``two-variable elliptic genus'' (cf. \cite{2varlib}, \cite{2vardijk}, \cite{2vareguchi}, \cite{2varhirzell}, \cite{2varkrichever},  \cite{2varwitten}).

\begin{introcor} (\ref{cor_jacobi_is_comm})
    The Jacobi orientation
    \[\MSU\to \mathrm{tmf}^{t\T}\]
    admits an $\EE_\infty$-lift.
\end{introcor}

Our final main result is an answer to two specific questions posed by Hahn-Yuan in \cite{hahnyuan}, which were aimed at advancing the program of understanding $\EE_\infty$ orientation theory. We refer the reader to \Cref{section_surprise_or} for more background and details.

\begin{introthm}(\ref{thm_both_MUP}, \ref{cor_forms_of_periodic}) Consider the three `forms of periodic complex bordism' given by $\MU^{t\T}$, $\MUP$, and $\MUP_\mathrm{snaith}$. Then $\MU^{t\T}$ admits $\EE_\infty$-orientations by both $\MUP$ and $\MUP_\mathrm{snaith}$. Therefore, $\Z^{t\T}$, a `form of periodic integral cohomology,' admits $\EE_\infty$-orientations by all three. 
    
\end{introthm}

\subsection{Conventions}

\begin{itemize}
\item For two spectra $X$ and $R$, we denote by $R(X)$ the spectrum of maps from $X$ to $R$ (which is usually denoted by $R^X$), i.e. the $R$-cochains on $X$.

\item We shall use the following notation for action of continuous maps on homotopies: Let $A,B$ be objects of an $\infty$-category $\cC$. If $h$ is a path in $\Map(A,B)$ (i.e. a homotopy between maps $A\to B$), then for $f\colon A' \to A$ we denote the image of $h$ under the map $f^*\colon \Map(A,B) \to \Map(A',B)$ by $h f$ and if $g\colon B \to B'$ is a map we shall denote the image of $h$ under $g_*\colon \Map(A,B) \to \Map(A,B')$ by $gh$.

%\item We use the adjectives ``commutative'' and ``$\EE_\infty$'' synonymously; in each instance our choice is completely aesthetic. 

\item Outside of \Cref{section_recollection}, orientations will always be implicitly $\EE_\infty$ unless otherwise stated.

\item For a commutative ring spectrum $R$, we denote by $\Perf(R)$ the $\infty$-category of perfect $R$-modules, i.e. these that can be built from $R$ by finite colimits, retracts, and desuspensions. 
In particular, $\Perf(\Sph)$ is the $\infty$-category of finite spectra. 

\item We let $\cat$ be the $\infty$-category of small $\infty$-categories and $\cat^\perf$ the $\infty$-category of small idempotent complete stable $\infty$-categories. All stable $\infty$-categories in this paper are idempotent complete by default. 
\end{itemize}

\subsection{Acknowledgments} We would like to thank: Eric Peterson for valuable feedback, and particularly for pointing out the work of Ando-French-Ganter \cite{afg}; Constantin Teleman for pointing out the work of Lawson et al \cite{lawsonetal}; Mark Behrens and Chris Schommer-Pries for a helpful discussion regarding \Cref{section_surprise_or}; and Yigal Kamel, Akhil Mathew, Charles Rezk, and Lior Yanovski for helpful conversations. 

S.C. is supported by a research grant from the Center for New Scientists at the Weizmann Institute of Science, and would like to thank the Azrieli Foundation for their support through an Early Career Faculty Fellowship.

\section{Recollection of complex orientations and characteristic classes}\label{section_recollection}

In this section we recall the basics of complex orientation theory, mainly to fix notation and conventions.

\subsection{Complex orientations}

\begin{defn}\label{defn_cx_or}
    Let $R$ be a homotopy ring spectrum. A \emph{complex orientation} of $R$ is one of the following equivalent pieces of data:
    \begin{enumerate}
        \item a homotopy ring map $\MU\to R$
        \item a factorization of the unit $\S\to R$ through the bottom cell $\S\to\Sigma^{\infty-2}\CP^\infty$
        \item a class $t\in R^2\CP^\infty$ which restricts to $\Sigma^21 \in R^2 S^2$ along $S^2 \simeq \CP^1\to\CP^\infty$ and induces an isomorphism $\pi_{-*}R[[t]]\to R^*\CP^\infty$.
        \item a functorial and additive\footnote{That is, $\theta_V\theta_W=\theta_{V\oplus W}$.} assignment of ``Thom'' classes $\theta_V\in R^n(\S^V)$ for every complex vector bundle $V\to X$ such that multiplication by $\theta_V$ induces an isomorphism $R^*X\simeq R^{*+n}(\S^V)$, where $\S^V$ denotes the Thom space of $V$.
    \end{enumerate}
    We shall refer to the class $t$ from Item 3 as the \emph{parameter} of the complex orientation, and to the isomorphisms of Item 4 as the \emph{Thom isomorphism} associated with the orientation.
    Note that by Item 3 and the K{\"u}nneth formula, $R^*(\CP^\infty \times \CP^\infty)\simeq \pi_{-*}R [[x,y]]$ where $x$ and $y$ are the pullbacks of $t$ along the two projections. Let $\mu:\CP^\infty \times \CP^\infty\to\CP^\infty$ be the multiplication map, which classifies the tensor product of line bundles. The \emph{formal group law} $F_R(x,y)$ is defined to be the power series associated to the element $\mu^*t$. When $R$ is clear from the context, we denote it by $F(x,y)$. 
\end{defn}

\begin{rem}\label{rem_Einfty_or}
    The majority of this paper is concerned with the much more structured notion of an $\EE_\infty$ complex orientation. By definition these are maps $\MU\to R$ of $\EE_\infty$-ring spectra. This is of course only defined when $R$ is an $\EE_\infty$-ring spectrum. Given item 1, there is an evident forgetful map from $\EE_\infty$ complex orientations to complex orientations.
\end{rem}

Let $R$ be a complex oriented homotopy ring spectrum, and let $V\to X$ be a complex vector bundle. Consider the associated projective bundle $\P(V)$, together with the map $\phi \colon \P(V)\to \CP^\infty$ classifying the tautological bundle on it. More precisely, if $L\to \CP^\infty$ is the tautological line bundle, then $\phi$ is the map such that $\phi^*L$ is the canonical one-dimensional sub-bundle of $p^*V$, where $p\colon \P(V) \to X$ is the projection. 

The class $t\in R^2\CP^\infty$ in Item 3 of \Cref{defn_cx_or} pulls back to a class $x = \phi^*t\in R^{2}\P(V)$ which produces (via e.g. the Leray-Hirsch theorem) an isomorphism 
\emph{of $R^*X$-modules}
\[
R^*\P(V)\simeq \bigoplus_{k=0}^{n-1} R^*(X)\cdot x^k.
\]
However, $x^{n}$ need not be zero, and its failure to be so is measured by characteristic classes. 

\begin{defn}\label{defn_chern}
    Let $R$ be a complex oriented homotopy ring spectrum, let $V\to X$ a vector bundle of dimension $n$, and let $x\in R^2\P(V)$ be defined as above. Define the \emph{Chern classes} 
    \[
    c_i(V)\in R^{2i}(X)
    \]
    by declaring $c_0(V)=1$, and then for $i\geq 1$, defining $c_i(V)$ to be the unique elements such that the equation
    \[0=\sum_{i=0}^{n}(-1)^ic_i(V)x^{n-i}\]
    holds in $R^*\P(V)\simeq\bigoplus_{k=0}^{n-1} R^*(X)\cdot x^k.$

    For a parameter $z$ of cohomological degree -2, define the \emph{total Chern class with respect to $z$}, $c(V,z)\in R^*(X)[[z]]$ to be the power series
    \[
    c(V,z):=\sum_{i=0}^nc_i(V)z^i.
    \]
\end{defn}

\begin{defn}\label{defn_Euler_class}
Let $R$ be a complex oriented homotopy ring spectrum and $V\to X$ a vector bundle of rank $n$. Define the \emph{Euler class}
\[e(V)\in R^n(X)\]
to be the pullback of $\theta_V\in R^n(\S^V)$ (Item 3 \Cref{defn_cx_or}) along the zero section $a_V\colon X\to \S^V$. It is a standard fact that this agrees with the $n$-th Chern class $c_n(V)$ of \Cref{defn_chern}.
\end{defn}

\begin{rem}\label{rem_line_Euler}
    Note that when $V\to X$ is a line bundle, classified by a map $\xi_V:X\to\CP^\infty$, then 
    \[
    e(V)=c_1(V)=\xi_V^*t\in R^2(X),
    \]
    where $t$ is as in \Cref{defn_cx_or}.
\end{rem}

\subsection{Characteristic series and characteristic classes}
Using the theory of Chern classes, one can classify all ``exponential'' characteristic classes valued in a complex oriented ring spectrum.
\begin{defn}
Let $R$ be a homotopy ring spectrum. An $R$-valued \emph{exponential characteristic class} is a map of $H$-spaces 
\[
c\colon \BU\to \GL_1(R).
\]
If $R$ is an $\EE_\infty$-ring spectrum, an \emph{$\EE_\infty$ exponential characteristic class} is a map of spectra 
\[
c\colon \bu \to \glone(R).
\]
\end{defn}
Thus, an exponential characteristic class assigns to every (rank zero, virtual) vector bundle $V\to X$ a class $c(V)\in R^0(X)^\times$, natural in maps of spaces and satisfying $c(V\oplus U) = c(V)c(U)$. We extend it to virtual vector bundles of non-zero rank by the convention that $c(\CC^n) = 1$ in $R^0(\pt)$. 

Assume that $R$ is complex oriented, with corresponding parameter $t\in R^2\CP^\infty$.
Given a characteristic class $c\colon \BU\to \GL_1(R)$, we can evaluate it on the tautological line bundle $L\to \CP^\infty$ to get a class 
\[
f(t) := c(L)\in R^0(\CP^\infty) \iso \pi_0(R[[t]]). 
\]
Since $c(\CC) = 1$ and $L|_{\pt} \simeq \CC$, we see that $f(0) = 1$. 
\begin{prop}[Splitting Principle]\label{splitting_princ}
Let $R$ be a complex oriented ring spectrum. 
Then, the association 
\[
c\mapsto f(t):= c(L)
\]
induces an isomorphism between the abelian group of isomorphism classes of $R$-valued exponential characteristic classes and the abelian group of power series $f(x) = 1 + a_1 x + a_2 x^2 + \dots$ with $a_k \in \pi_{2k} R$, endowed with the multiplicative group structure. 
\end{prop}

\begin{defn}\label{f_chern_class}
Let $R$ be a complex oriented ring spectrum with associated parameter $t\in R^2 \CP^\infty$. 
We denote the inverse of the above association by 
\[
f\mapsto c^f,
\]
taking a power series $f(x) = 1 + a_1 x + a_2 x^2 + \dots$ to the unique $R$-valued exponential characteristic class $c^f$ satisfying $c^f(L) = f(t) \in R^0\CP^\infty$. 
\end{defn} 

\begin{rem} \label{rem:splitting_explicit}
One can explicitly recover $c^f$ from $f$. Namely, let $V\to X$ be a vector bundle of rank $n$ over a space $X$. We can expand the multivariate power series $\prod_{i=1}^n f(x_i)$ uniquely as a power series in the elementary symmetric polynomials 
\[
\prod_{i=1}^n f(x_i) = G(\sigma_1,\sigma_2,\dots,\sigma_n)
\]
where 
\[
\sigma_k = \sum_{I\subseteq [n], |I| = k} \prod_{i\in I} x_i. 
\]
Then, using the splitting principle one easily shows that 
\[
c^f(V) = G(c_1(V),\dots,c_n(V)).
\]
\end{rem}

\begin{rem}
Note that $c^f$ depends on the complex orientation of $R$ through the choice of parameter $t\in R^2\CP^\infty$ and not only on the abstract power series $f(x)$. 
\end{rem}

%\begin{defn}\label{defn_symm_fun_chern}
\begin{example}\label{exmp_chern_series_for_Z} 
Let $R$ be a complex oriented ring spectrum, and consider the $R$-algebra $R[z]$ with $|z| = 2$. Then, the $\pi_*R[z]$-valued power series $f(x) = 1 + zx$ defines a characteristic class 
\[
c^f\colon \BU \to \GL_1(R[z]). 
\]
In this case, the corresponding multivariate power series is 
\[
G(\sigma_1,\dots,\sigma_n) = 1 + \sigma_1 u + \dots + \sigma_n z^n
\]
and hence 
\[
c^f(V) = \sum_k c_k(V)z^k = c(V;z)
\]
is the total Chern class from \Cref{defn_chern}.
\end{example}

\begin{example}\label{Euler_tate_series} (cf. \Cref{section_sharp})
Consider now $R((z^{-1}))$ with $|z| = 2$, let $F(x,y)$ be the formal group law associated with the complex orientation of $R$ and let 
\[
f(x) = z(x +_F z^{-1}) = 1 + xz + \dots.
\]
We obtain a canonical characteristic class 
\[
c^f\colon \BU\to \GL_1(R((z^{-1}))).
\] 
\end{example}

The characteristic class in the \Cref{Euler_tate_series} is the main object of study in this paper. Observe that when $F$ is the additive formal group law, it agrees with the extension of scalars of the total Chern class along the canonical map of homotopy ring spectra $R[z]\to R((z^{-1}))$. 

\subsection{Characteristic classes and quotients of orientations}
Exponential characteristic classes and complex orientations of a complex oriented ring spectrum $R$ are closely related. 
\begin{defn}\label{def_plain_or_divis}
Given two complex orientations $\omega$ and $\omega'$ of a homotopy ring spectrum $R$ with associated Thom isomorphisms $\theta$ and $\theta'$, we obtain an exponential characteristic class 
\[
\frac{\omega'}{\omega}\colon \BU \to \GL_1(R) 
\]
taking a vector bundle $V\to X$ to the 
 image of $1\in R^0(X)$ under the composite  
\[
R(X) \oto{\theta_{V}'} R(\Th(V)) \oto{\theta_V^{-1}} R(X). 
\]  
\end{defn}

When $R$ is an $\EE_\infty$-ring spectrum with $\EE_\infty$-orientations (cf. \Cref{rem_Einfty_or}) $\omega$ and $\omega'$, the quotient refines to a map $\bu \to \glone(R)$ that we denote again by $\frac{\omega'}{\omega}$. It can be defined alternatively as the mate of the composite 
\[
\Sph[\BU] \to \MU[\BU]\simeq \MU\otimes \MU \oto{\omega' \otimes \omega} R
\]
under the adjunction
\[
\Sph[-]\colon \Sp^\cn \adj \calg(\Sp) : \glone . 
\]

\section{The Euler orientation}\label{eul_or_cats}

\subsection{Orientations of categories}

%For $\FF= \CC$ or $\RR$, 
Let $\J$ be the $\infty$-category associated with the topological category of complex vector spaces and injective linear maps, so that 
\[
\J^\simeq \simeq \bigsqcup_n B\GL_n(\C).
\]
Endow $\J$ with the symmetric monoidal structure given by the direct sum of vector spaces. 
Then there is a symmetric monoidal functor $\Sph^{(-)}\colon \J\to \Perf(\Sph)$ taking a vector space $V$ to the stable one-point compactification $\Sph^V$. Moreover, $\J^\simeq$ acquires the structure of a commutative monoid in spaces.
Let 
\[
    \ku := (\J^\simeq)^{\gp}
\] 
be its group completion, which is the connective topological $K$-theory spectrum of $\C$. 
%Hence, if $\FF=\RR$ we have $\ku = \ko$ and if $\FF=\RR$ we have $\ku = \ku$. 
Denote by
$
\bu
$
the connected cover of $\ku$. By Bott periodicity, we have $\bu \simeq \Sigma^2\ku$.

The functor $\Sph^{(-)}$ restricts to a map $\J^\simeq \to \pic(\Sph)$ of commutative monoids in spaces. Since the target is group complete, it extends to map from $\ku$. This is the $j$-homomorphism $j\colon \ku \to \pic(\Sph)$. 

We reformulate the concept of an orientation of an $\EE_\infty$-ring spectrum in terms of stable categories. 
\begin{defn}\label{defn_cat_or}
%{\red Kiran: Is there a reason to include the ``factor through $\pi_0$" versions of all the orientation definitions? If not, it may increase readability to just omit it and stick with the ``nullhomotopy" version?} 
Let $\cC$ be a stably symmetric monoidal $\infty$-category. A \emph{complex orientation} of $\cC$ is a factorization of the form  
\[
\xymatrix{
\ku \ar[d]\ar^{j}[r]  & \pic(\Sph)\ar[d] \\
\pi_0\ku  \ar@{..>}[r]\ar@{=>}^{\omega}[ru]   & \pic(\cC)
}.
\]
Equivalently, a complex orientation of $\cC$ is the data of a null-homotopy of the composite 
\[
\bu  \to \ku  \oto{j} \pic(\Sph) \oto{u} \pic(\cC). 
\]
\end{defn} 

We denote the space of complex orientations of $\cC$ by $\Or(\cC)$. These spaces assemble to a functor 
\[
\Or \colon \calg(\cat^\perf) \to \Spc.  
\]
from the $\infty$-category of idempotent complete stably symmetric monoidal $\infty$-categories.

This categorical version of orientation theory is really just a reformulation of the classical $\EE_\infty$-orientation story, as summarized by the following proposition. 
\begin{prop}
For $\cC \in \calg(\cat_\st)$, the following spaces are equivalent. 
\begin{itemize}
\item The space $\Or(\cC)$. 
\item The space of null-homotopies of the composite 
\[
\bu \to \ku \oto{j} \pic(\Sph) \to \pic(\cC). 
\]
\item The space of $\EE_\infty$ complex orientations of the $\EE_\infty$-ring spectrum $\End(\one_\cC)$. 
\end{itemize}
\end{prop}

\begin{proof}
The first and second spaces are equivalent by the cofiber sequence 
\[
\bu \too \ku \too \pi_0 \ku.
\]
For the equivalence between the second and third items, note that a complex orientation of $\End(\one_{\cC})$ is the same data as a null-homotopy of the composite $\bu \to \pic(\Sph) \to \pic(\End(\one_\cC))$ (cf. e.g. \cite{abghr}, \cite{barthelthom}). Thus, the claim follows from the fact that $\bu$ is connected and the map $\pic(\End(\one_\cC)) \to \pic(\cC)$ is $0$-truncated. 
\end{proof}
\subsection{An equivariant version of $\Sph^{(-)}$}

Let $\T$ be the Lie group of unit complex numbers. The grouplike commutative monoid $B\T$, classifying $1$-dimensional vector spaces\footnote{Again, we consider maps between such vector spaces with the topology induced from $\C$.} over $\C$ acts on $\J$ by the tensor product. Via this action, $\J$ is promoted to an object of $\calg(\cat)^{B^2\T}$. Let $e\colon  \pt \to  B^2\T$ be the inclusion of the basepoint. We have an adjunction
\[
e^*\colon \calg(\cat)^{B^2\T} \adj \calg(\cat): e_*
\]
where $e^*$ forgets the action of $B\T$ and      
$
e_*\cC \simeq \cC^{B\T}
$ 
with the translation action of $B\T$. 

\begin{defn}\label{defn_equi_j}
The non-equivariant functor $\Sph^{(-)}\colon e^*\J \to \Perf(\Sph)$
corresponds via the adjunction above to an equivariant functor that we abusively denote again by  
\[
\Sph^{(-)}\colon \J \to e_*\Perf(\Sph) \simeq \Perf(\Sph)^{B\T}.
\]
Explicitly, this functor takes a vector space $V$ to the stable spherical fibration $\Sph^{V\otimes L}$ over $B\T \simeq \CP^\infty$, where $L$ is the tautological bundle. Alternatively, it sends $V$ to the finite-spectrum-with-$\T$-action given by $\S^V$ and its action of $\T$ by scalar multiplication.
\end{defn}

Upon group completion we get a version of the $j$-homomorphism, i.e. a morphism of spectra 
\[
\tilde{j} \colon \ku \to \pic(\Sph)^{B\T} \simeq \pic(\Perf(\Sph)^{B\T}).  
\]

%\begin{rem}
%Let $\rho \colon \TT \to \GL_1(\CC)$ be the standard representation of $\TT$, regarded as a vector bundle over $B\TT$.  
%Unwinding the definitions, naturally in a vector space $V$ over $\CC$ we have 
%\[
%\tilde{j}_\CC(V) = \Sph^{V\otimes \rho} \in \pic(\Perf(\Sph)^{B\TT}). 
%\]
%\end{rem}

\subsection{Projective orientations} 

Using $\tilde{j}$, we shall now define a variant of the categorical orientation theory presented above which takes the action of $B\T$ into account.
\begin{defn}
Let $\cC$ be a stably symmetric monoidal $\infty$-category endowed with an exact, symmetric monoidal functor $\phi\colon \Perf(\Sph)^{B\T} \to\cC$. A \emph{projective complex orientation} of $(\cC,\phi)$ is a factorization of the form  
\[
\xymatrix{
\ku \ar[d]\ar^{\tilde{j}}[r]  & \pic(\Sph)^{B\T}\ar^\phi[d] \\
\pi_0\ku  \ar@{..>}[r] \ar@{=>}^{\omega}[ru]   & \pic(\cC)
}
\]
Equivalently, it is the data of a null-homotopy of the composite 
\[
\bu  \to \ku  \oto{\tilde{j}} \pic(\Sph)^{B\T} \oto{\phi} \pic(\cC). 
\]
\end{defn}
We denote the space of projective complex orientations of $(\cC,\phi)$ by $\Or^\proj(\cC,\phi)$. These assemble to a functor 
\[
\Or^\proj \colon \calg_{\Perf(\Sph)^{B\T}}(\cat^\perf) 
\to \Spc.
\]

\begin{rem}
As for non-equivariant orientations, the notion of a projective orientation is defined in categorical terms but the information can be reduced quite a bit. Namely, since $\bu $ is connected, the data of an equivariant complex orientation of $(\cC,\phi)$ is the same as a nullhomotopy of the map 
\[
\Omega \ku  \oto{\Omega \phi\circ \tilde{j}} \Omega \pic(\cC) \simeq \glone(\one_\cC).
\]
\end{rem}

\begin{rem}
    The key examples we have in mind are such that $\one_\cC = R^{t\T}$ for a commutative ring spectrum $R$ endowed with the trivial $\T$-action. As a first approximation, one may think of the case $\cC = \Mod(R^{t\T})$ with the functor $\Perf(\Sph)^{B\T} \to \Mod(R^{t\T})$ taking $X$ to $(R\otimes X)^{t\T}$. Since this functor is in general not symmetric monoidal, we shall work with a modification of this picture, replacing modules over $R^{t\T}$ with a categorified version of Tate's construction $\Perf(R)^{t\T}$. 
\end{rem}

\begin{rem}
    Just like in the case of ordinary complex orientations, there is a universal object $(\cC_\mathrm{univ}, \phi_\mathrm{univ})$ with a projective orientation: by \cite[Theorem 7.13]{carmeliramzi} we have $\cC_\mathrm{univ}=\mathrm{Mod}_{\mathrm{Sp}^{B\T}}(\mathrm{Th}(\tilde{j}))$ where $\mathrm{Th}(\tilde{j})\in\mathrm{Sp}^{B\T}$ is the Thom spectrum of $\tilde{j}$, and $\phi_\mathrm{univ}$ is the evident base change functor. A projective orientation of $(\cC,\phi)$ is then equivalent to a symmetric monoidal functor (over $\Sp^{B\T}$) $\cC_\mathrm{univ}\to \cC$, which in particular consists of a ring map $\mathrm{Th}(\tilde{j})^{h\T}\to\one_\cC$. Note that the underlying spectrum of $\mathrm{Th}(\tilde{j})$ is $\MU$. This has the somewhat surprising consequence (cf. \Cref{Euler_transf_becomes_iso_after_tate}) that $\MU^{h\T}$ admits a ring map to $\S^{tC_p}$ so that (after $p$-completion) it contains $\S_p$ as a summand. 
    %and one could in principle compute stable homotopy groups from $\MU$ via a homotopy fixed points spectral sequence. 
    We refrain from giving details as we do not pursue this further here.
\end{rem}

\begin{rem}\label{rem:or_torsor}
Just like for non-equivariant orientations, since $\Or^\proj(\cC)$ is the space of nullhomotopies of a fixed map $\bu  \to \pic(\cC)$, it is a torsor over the connective spectrum of maps 
\[
\hom(\bu , \Omega \pic(\cC)) \simeq \hom(\bu , \glone(\one_\cC)).
\] 
We denote the corresponding action of $\hom(\bu , \glone(\one_\cC))$ on $\Or^{\proj}(\cC)$ by 
\[
(\alpha,\omega)\mapsto \alpha \cdot \omega.
\]
\end{rem}

In view of \Cref{rem:or_torsor}, we can define the following.

\begin{defn}\label{defn_division_or} (Compare with \ref{def_plain_or_divis})
Let $\cC \in \calg_{\Sp^{B\T}}(\cat^\perf)$.  
Given two projective complex orientations $\omega,\omega' \in \Or(\cC)$, their \emph{quotient} 
\[
\frac{\omega'}{\omega}\colon \bu  \to \glone(\one_\cC)
\] 
is the unique (up to a contractible space of choices) map such that 
\[
\frac{\omega'}{\omega} \cdot \omega \simeq  \omega'.
\]
\end{defn}

Finally, we show how to turn non-equivariant orientations into equivariant ones.

\begin{defn}\label{defn_proj_or}  
Let $\cC\in \calg(\cat^\perf)$ and  write $u:\pic(\Sph) \to \pic(\cC)$ for the unit map. 
The adjunction equivalence 
\[
\Map(\bu,\pic(\cC))\simeq \Map_{B\T}(\bu,\pic(\cC)^{B\T}) 
\]
carries the map $u\circ j$ to the map $u^{B\T}\circ \tilde{j}$, and hence induces a map on orientation spaces 
\[
(-)^\proj\colon \Or(\cC) \to \Or_{\CC}^\proj(\cC^{B\T},u^{B\T}).
\]
This map takes an orientation $\omega$ to the 
%($B\T$-equivariant) 
projective orientation $\omega^\proj$ described as follows. The fixed point functor sends the nullhomotopy $\omega$ to a nullhomotopy $\omega^{B\T}$ of 
\[\bu^{B\T} \oto{j^{B\T}} \pic(\Sph)^{B\T} \oto{u^{B\T}} \pic(\cC)^{B\T}.
\]
Then $\omega^\proj$ is defined to be the induced nullhomotopy of the composite
\[
\bu \to \bu^{B\T} \oto{j^{B\T}} \pic(\Sph)^{B\T} \oto{u^{B\T}} \pic(\cC)^{B\T},
\]
where the first map is the mate of the action map $B\T \otimes \bu \to \bu$ induced from the action of $B\T$ on $\J^\simeq$ by passing to group completion and connected cover:
\[
\omega^{\proj}:\xymatrix{
\bu\ar[d] \ar@/^2.0pc/^{\tilde{j}}[rr] \ar[r]& \ku^{B\T}\ar[d] \ar^{j^{B\T}}[r] &  \pic(\Sph)^{B\T} \ar[d]   \\ 
0 \ar@{=}[r]\ar@{=}[ru] & 0\ar@{=>}^{\omega^{B\T}}[ru]\ar[r] &  \pic(\cC)^{B\T} 
}.
\]
\end{defn}

%\begin{rem}
%Let $\pi\colon B\T \to \pt$ be the terminal map, inducing a morphism of connective spectra $\hom(\bu, \glone(\one_\cC)) \to \hom(\bu, \glone(\one_\cC)^{\bu})$. 
%Then the map $(-)^{\proj}$ defined above intertwines the action of $\hom(\bu,\glone(\one_\cC))$ on $\Or_\C(\cC)$ with the action of $\hom(\bu,\glone(\one_{\cC})^{B\T})$ on $\Or_\C^{B\T}$: 
%\[
%\pi^*\alpha\cdot \omega^\proj \simeq (\alpha \cdot \omega)^{\proj} \quad \forall \alpha \colon \bu \to \glone(\one_\cC),\quad \omega \in \Or_\C(\cC).  
%\]
%\end{rem}

\subsection{The Euler map}
Endow the $\infty$-category $\Fun(\J,\Perf(\Sph)^{B\T})$ with the Day convolution symmetric monoidal structure, so that commutative algebras in it are lax symmetric monoidal functors. Since the unit of $\J$ is initial, the constant functor $\cnst{\Sph}\colon \J \to \Perf(\Sph)^{B\T}$ which takes all the objects of $\J$ to the sphere spectrum with constant $\T$-action is the unit of $\Fun(\J,\Perf(\Sph)^{B\T})$. 

\begin{defn}
Since $\cnst{\Sph}$ is the unit of $\Fun(\J,\Perf(\Sph)^{B\T})$, there is a unique  natural transformation
\[
a_{(-)}\colon \cnst{\Sph} \to \Sph^{(-)}.
\]
of symmetric monoidal functors. 
We refer to $a_{(-)}$ as the \emph{Euler transformation}. 
For $V\in \J$, the $\T$-equivariant map $a_V\colon \Sph \to \Sph^V$ is obtained by applying $\Sph^{(-)}$ to the inclusion $\{0\}\into V$. 
\end{defn}

%\subsection{The Tate Construction and Inverting $a_{(-)}$} 
Next, we discuss a categorical variant of the Tate fixed points functor for $\T$.
\begin{prop}
Let $\cC \in \cat^\perf$. 
There is a fully faithful embedding 
\[
\Nm_\T \colon \cC[B\T] \into \cC^{B\T} 
\]
from the constant $B\T$-shaped colimit to the constant $B\T$-shaped limit in $\cat^\perf$, with essential image the thick subcategory of $\cC^{B\T}$ generated by the induced local systems, i.e., the objects of the form $X\otimes \T$ for $X\in \cC$ (and with the action coming from the translation action of $\T$ on itself). If $\cC$ is stably symmetric monoidal then $\Nm_\T$ is the inclusion of a thick tensor ideal, so that the quotient 
\[
\cC^{t\T} := \cC^{B\T} / \cC[B\T]
\]
admits a canonical commutative $\cC^{B\T}$-algebra structure in $\cat^\perf$. In this case, $\End(\one_{\cC^{t\T}}) \simeq \End(\one_\cC)^{t\T}$, where the right hand side is the classical Tate fixed points construction on $\End(\one_\cC)$.  
\end{prop}

\begin{proof}
We have an embedding 
\[
\cC[B\T] \into \Ind(\cC[B\T])\simeq \Ind(\cC)^{B\T}. 
\]
To identify its essential image, note that $\cC[B\T]$ is generated from the image of the push-forward functor $\cC \to \cC[B\T]$ along the inclusion $\pt \into B\T$, whose image under $\Nm_\T$ is exactly the induced objects $X\otimes \T$ in $\Ind(\cC)^{B\T}$. In particular, since $\T$ is a finite space, the essential image lies in $\cC^{B\T} \subseteq \Ind(\cC)^{B\T}$.  

When $\cC$ is symmetric monoidal, we have that $Y\otimes (X\otimes \T) \simeq (Y\otimes X) \otimes \T$ for $Y\in \cC^{B\T}$ and $X\in \cC$ (regarding $Y\otimes X$ as an object of $\cC$), showing that $\cC[B\T]$ is a thick tensor ideal. Finally, the determination of the endomorphism object of the unit in $\cC^{t\T}$ follows from the same consideration as in the proof of \cite[Theorem I.4.1 (iv)]{ns}.
\end{proof}

We denote the quotient functor from the proposition above by $T\colon \cC^{B\T} \to \cC^{t\T}$, so that if $\cC$ is symmetric monoidal so is $T$.

\begin{prop}\label{Euler_transf_becomes_iso_after_tate}
The image of the Euler transformation $a_{(-)}$ under the symmetric monoidal functor $T\colon \Perf(\Sph)^{B\T} \to \Perf(\Sph)^{t\T}$ is a natural isomorphism.   
\end{prop}

\begin{proof}
We have to show that for every $V\in \J$ the fiber of the map $a_V \colon \Sph \to \Sph^V$ belongs to the image of $\Nm_\T\colon \Perf(\Sph)[B\T] \to \Perf(\Sph)^{B\T}$. A choice of metric on $V$, with unit sphere denoted by $S(V)$, identifies this fiber as the suspension spectrum $\Sigma^\infty_+S(V)$. But $S(V)$ is a finite $\T$-CW space with free action, so its suspension spectrum belongs to the thick subcategory generated by the induced $\T$-equivariant objects of $\Perf(
\Sph)$.
\end{proof}

We can now use the Euler transformation to define a projective complex orientation of $\Perf(\Sph)^{t\T}$.

\begin{defn}\label{defn_Euler_or} Note that the functor $\cnst{\Sph}\colon \J \to \Perf(\Sph)^{B\T}$, after restricting to $\J^\simeq$, extends to the $0$-map $\ku  \to \pic(\Sph)^{B\T}$. Therefore,
by \Cref{Euler_transf_becomes_iso_after_tate} the Euler transformation $a_{(-)}$ induces a null-homotopy 
\[
Ta_{(-)}\colon 0 \iso T\circ \tilde{j}, 
\]
and in particular a null-homotopy of the restriction to $\bu$:
\[
\xymatrix{
\bu  \ar[d]\ar^{j}[r] &      \pic(\Sph)^{B\T}    \ar^{\wr}[d]     \\
\ku \ar[d] &\ \pic(\Perf(\Sph)^{B\T})\ar^{T}[d]  \\ 
0 \ar[r]\ar^{a_{(-)}}@{=>}[ruu]       &\ \pic(\Perf(\Sph)^{t\T})  
}.
\]
We thus obtain a projective complex orientation of $(\Perf(\Sph)^{t\T}, T)$ that we denote by 
\[
\ortate\in \Or^\proj(\Perf(\Sph)^{t\T}, T).
\]
\end{defn}

\section{Coherent commutativity for characteristic classes}
In this section we use the Euler orientation constructed in the previous section to construct examples of $\EE_\infty$ characteristic classes valued in Tate-fixed points ring spectra $R^{t\T}$.

\subsection{The Tate-valued Chern class}
%Again, letting $\FF = \RR$ or $\CC$, f
From now on we fix a complex oriented stably symmetric monoidal $\infty$-category $\cC \in \calg(\cat^\perf)$. The most important example is $\cC = \Perf(R)$ for a complex oriented commutative ring spectrum $R$.  
%\footnote{Recall that our convention outside of \Cref{section_recollection} is that orientations are implicitly $\EE_\infty$.} 
%$\EE_\infty$-ring spectrum $R$ with unit map $u:\S\to R$. 
Let $u\colon \Perf(\Sph) \to \cC$ be the unit functor, and consider the functor $u^{B\T} \colon \Perf(\Sph)^{B\T} \to \cC^{B\T}$. 
%Write $u^{B\T} :\Perf(\Sph)^{B\T} \to \Perf(\Sph)^{B\T}$ for the map induced by the unit $\S\to R$. Denote by $\omega\in \Or(R)$ its chosen complex orientation. 
From $\omega$ we can now obtain a projective orientation (\Cref{defn_proj_or}) 
\[
\omega^\proj \in \Or^{\proj}(\Perf(R)^{B\T}, u^{B\T}).
\]
This orientation can be pushed forward along the Tate-quotient functor $T\colon \Perf(R)^{B\T}\to \cC^{t\T}$ to obtain the projective orientation 
\[
T(\omega^\proj)\in \Or^{\proj}(\cC^{t\T}, T\circ u^{B\T}). 
\]
On the other hand, we can push $\ortate$ (\Cref{defn_Euler_or}) along the evident functor $u^{t\T}:\Perf(\Sph)^{t\T} \to \cC^{t\T}$ to obtain a \emph{second} projective orientation
\[
u^{t\T}(\ortate) \in \Or^{\proj}(\cC^{t\T}, T\circ u^{B\T}). 
\]
The situation is summarized in the following diagram 
\begin{equation} \label{diag:two_orientations}
\xymatrix{ 
\ku\ar^{\tilde{j}}[rd]\ar[dd] \ar[rr]  & &   \pi_0\ku\ar[d] \ar@{=>}_{\ortate}[ld] \\     
      & \pic(\Sph)^{B\T}\ar^{u^{B\T}}[d]\ar_T[r] & \pic(\Perf(\Sph)^{t\T})\ar^{u^{t\T}}[d] \\
 \pi_0\ku \ar[r] \ar@{=>}^{\omega^\proj}[ru]    & \pic(\cC)^{B\T} \ar_T[r]   & \pic(\cC^{t\T})
}
\end{equation}

\begin{defn}\label{defn_tch}
Let $R$ be an $\EE_\infty$-ring spectrum with $\EE_\infty$-complex orientation $\omega$. 
We define the \emph{Euler-Tate characteristic class} to be the the quotient (\Cref{defn_division_or}) of these two projective orientations: 
\[
\tch := \frac{T(\omega^\proj)}{u^{t\T}(\ortate)}\colon \bu \to \glone(\one_{\cC^{t\T}}). 
\]
\end{defn}

Let $R = \End(\one_\cC)$, so that $\End(\one_{\cC^{t\T}}) = R^{t\T}$.  The Euler-Tate characteristic class is an exponential characteristic class that is canonically $\EE_\infty$-multiplicative, as it is given as a map of spectra (not just of $H$-spaces, for example). Now, being an exponential $R^{t\T}$-valued characteristic class, it has a characteristic series $f(x) \in \pi_*R^{t\T}[[x]]$. Our next goal is to identify this characteristic series, and deduce that certain key examples of exponential characteristic classes admit canonical $\EE_\infty$-refinements. First, we identify the value of $\tch$ on an arbitrary vector bundle. Note that $\pi_*R^{t\T} \simeq \pi_*R((t))$ as an $\pi_*R^{B\T} \simeq \pi_*R[[t]]$-algebra for $t\in R^2B\T$ the parameter associated with the complex orientation of $R$.  

From now on we only consider orientations of commutative ring spectra for clarity, but the results translate immediately to the more general setting of categories using the functors $\Perf(\End(\one_\cC)) \to \cC$. 

\begin{prop}\label{prop:Euler_Tate_on_vb}
Let $R$ be an $\EE_\infty$-ring spectrum with $\EE_\infty$-complex orientation $\omega$. Let $V\to X$ be a complex vector bundle of dimension $n$ over a space $X$. Consider the vector bundle $V\otimes L$ on $X\times \CP^\infty \simeq X\times B\T$. Then    
\[
\tch(V) = \frac{e(V\otimes L)}{e(L)^{n}} \quad \in \quad \pi_*R((t))[[x]] 
\]
where $e(-)$ denotes the Euler class from \Cref{defn_Euler_class}.
\end{prop}

\begin{proof}
The characteristic class $\tch$ is defined as the quotient of the projective orientations $T(\omega^\proj)$ and $u^{t\T}(\ortate)$. Hence, to compute $\tch$ on a given vector bundle we need to compose the Thom isomorphism associated with $T(\omega^\proj)$ and the inverse of the Thom isomorphism associated with $u^{t\T}(\ortate)$.
On cohomology, the first isomorphism is gotten from the Thom isomorphism of the orientation $\omega$ (and inverting $t$)
\[
\theta_{V\otimes L}:R^{*}(X\times B\T)[t^{-1}] \simeq  R^{*+n}(\Sph^{V\otimes L})[t^{-1}]
\]
and the second from pullback along the Euler map 
\[
a_{V\otimes L}^* \colon R^*(\Sph^{V\otimes L})[t^{-1}] \iso R^*(X \times B\T)[t^{-1}].
\] 
Hence the composite 
\[
 R^{*}(X\times B\T)[t^{-1}] \oto{\theta_{V\otimes L}} R^{*+n}(\Sph^{V\otimes L})[t^{-1}]     \oto{a_{V\otimes L}^*}  R^{* + n}(X\times B\T)[t^{-1}]
\]
is (by definition) multiplication by the Euler class $e(V\otimes L)$, which is the term appearing in the numerator of the proposed formula. By our convention regarding the value of characteristic classes on bundles of non-zero dimension, the function $\tch$ is defined by taking the quotient of the above composite with the analogous one for the trivial bundle $\CC^n$ , which gives the $e(L)^n$-factor in the denominator.  
\end{proof}

\begin{rem}
The denominator appearing in the formula for $\tch$ is a byproduct of the fact that $\tch$ admits a finer structure: is corresponds to a $\ku$ (rather than $\bu$) orientation of $R^{t\T}$. We will return to this point in \Cref{section_sharp}. 
\end{rem}

\begin{prop}\label{prop_char_series_tch}
Let $R$ be an $\EE_\infty$-ring spectrum with $\EE_\infty$-orientation $\omega$, associated parameter $t\in R^2(B\T)$ so that $\pi_*R^{t\T} \simeq \pi_*R((t))$, and let $F$ be the formal group law associated with $\omega$, denoted $(u,v)\mapsto u +_F v$. Let $x\in (R^{t\T})^2\CP^\infty$ be the image of the parameter of the orientation of $R$ under the unit map $R\to R^{t\T}$, so that $(R^{t\T})^*\CP^\infty=R_*((t))[[x]]$. Then the characteristic series of $\tch$ is given by the power series 
\[
f(x,t) = \frac{(x +_F t)}{t}\in R_*((t))[[x]]. 
\]
\end{prop}

\begin{proof}
By definition, $f(x,t)$ is the power series obtained by evaluating $\tch$ on the tautological bundle of $\CP^\infty$. To distinguish the two appearances of $L$ (one from the definition of projective orientations and the other from the definition of the characteristic series), let us denote by $L_1$ the tautological bundle on $\CP^\infty$ (whose parameter is $x$) and by $L_2$ the tautological bundle on $B\T$ (whose parameter we denote by $t$), so that $e(L_1) = x$ and $e(L_2) = t$. By the definition of the formal group law $F$, we then have 
\[
e(L_1 \otimes L_2) = x +_F t \in R^*(\CP^\infty \times B\T) \simeq \pi_*R[[x,t]].
\]
Hence, by \Cref{prop:Euler_Tate_on_vb} we have
\[
\tch(L_1) = \frac{e(L_1 \otimes L_2)}{e(L_2)} = \frac{x +_F t}{t}
\]
as desired.
\end{proof}

\begin{cor}\label{cor:coherent_chern}
    When $R=\Z$, the Euler-Tate characteristic class
    \[\tch: \bu\to\glone(\Z^{t\T})\]
    is an $\EE_\infty$-refinement of the total Chern class with respect to $t^{-1}$. That is, after identifying $\pi_*\Z^{t\T}=\Z((t))$, $|t|=2$,  
    \[
    \tch(V)= \sum_{i=0}^\infty c_i(V)t^{-i}.
    \]
\end{cor}
\begin{proof}
    Note that when $F(x,y)=x+y$, the series $t^{-1}(x+_Ft)$ generates the total Chern class (cf. \Cref{exmp_chern_series_for_Z}). Thus we apply \Cref{prop_char_series_tch} with $R=\Z$ to see that $\tch$ is a lift of the total Chern class.
\end{proof}

\begin{rem}
The question of  whether the total Chern class admits an $\EE_\infty$-refinement goes back to Segal (\cite{segal}). Such a refinement was constructed by Lawson et al in \cite{lawsonetal} using, as the target spectrum, geometrically constructed infinite loop spaces built out of algebraic cycles in $\CP^n$. 
On the other hand, \Cref{cor:coherent_chern} provides such a lift with target $\glone(\Z^{t\T})$ using the Euler-Tate characteristic class instead. We expect that these two lifts are in fact the same (up to isomorphism), and in particular that the spaces of algebraic cycles in $\CP^n$ for $n\to \infty$ stabilizes to $\glone(\Z^{t\T})$, but we shall not consider this question further here. 
%{\red Change this if we eventually do :-)}
\end{rem}

\section{The sharp construction and the Jacobi orientation}\label{section_sharp}
In this section we use our theory to show that the ``\#''-construction of Ando-French-Ganter (\cite{afg}) lifts to an operation of $\EE_\infty$-orientations. Within the commutative setting, we also generalize their construction to apply to a larger class of rings, thereby producing a rich source of new $\EE_\infty$-orientations. In particular, this allows us to construct an $\EE_\infty$-refinement of the Jacobi orientation. We begin by recalling their setup.

Throughout this section, unless otherwise stated, $E$ will denote an \emph{even-periodic} homotopy ring spectrum and $n$ a non-negative integer. Such an $E$ always admits a homotopy complex orientation\footnote{A homotopy complex orientation always exists because the Atiyah-Hirzebruch spectra sequence for $E^*\CP^\infty$ collapses for degree reasons.} (i.e. homotopy ring map) $\MU\to E$. Recall that $\MU\l n\r$ is defined to be the Thom spectrum of $j$-homomorphism restricted to the $n$-th connective cover of $\ku$. For example, $\MU\l 0\r=\mathrm{MUP}$, $\MU\l2\r=\MU$, and $\MU\l 4\r=\MSU$.  

\begin{defn}\label{Cp_structure}
    Write $L_i$ for the pullback of the universal bundle $L$ along the $i$th projection map $(\CP^\infty)^n\to\CP^\infty$. A \emph{$C^n$-structure on $E$}, for $n>0$, is a choice of Thom class for 
    \[\Lambda(n):=\bigotimes_{i=1}^n (1-L_i)\rightarrow (\CP^\infty)^{\times n}\]
    which is invariant under the evident $\Sigma_p$-action, and which satisfies a ``cocycle condition'': write $d_0,d_1,..,d_n$ for the maps $(\CP^{\infty})^{n+1}\to(\CP^{\infty})^{n}$ occurring in the bar complex\footnote{So e.g. $d_0$ is the projection on to the last $n$ coordinates, $d_1$ multiplies the first two coordinates, $d_2$ multiplies the second two,..., and $d_n$ projects onto the first $n$ coordinates.} of the group $\CP^\infty$. The alternating sum of the pullbacks of $\Lambda(n)$ along the $d_i$ is canonically the rank 0 trivial bundle, so one requires that the same alternating sum of pullbacks of the chosen Thom class of $\Lambda(n)$ is $1$. 
    
    For $n=0$ a $C^0$-structure is a choice of ``generalized Thom class,'' namely an element $y\in E^0\CP^\infty$ such that the restriction of $y$ to the bottom cell $S^2$ is an invertible element $u^{-1}\in\pi_{2}E$. In particular, $uy$ is a Thom class for $L$.

    We write $C^n(E)$ for the set of $C^n$-structures on $E$.

\end{defn}

\begin{rem}
    This is an abbreviated form of the standard terminology---one usually speaks about a $C^n$-structure on a specified line bundle over the formal group $\G_E=\mathrm{Spf}(E^*\CP^\infty)$. The line bundle we have in mind is the one associated to the augmentation ideal $\mathrm{ker}(E^*\CP^\infty\to E^*(\mathrm{pt}))$.
\end{rem}

\begin{thm}(\cite{ahs})\label{AHS_Cp_to_or}
    Write $[\MU\l 2n \r,E]_\mathrm{hRing}$ for the set of homotopy classes of homotopy ring maps $\MU\l 2p\r\to E$. Then the map
    \[[\MU\l 2n \r,E]_\mathrm{hRing}\oto{}C^p(E)\]
    given by sending a homotopy orientation to the associated Thom class of $\Lambda(n)$ ($n>0$) or ``generalized Thom class'' ($n=0$) is an isomorphism when $n=0,1,2,3$.
\end{thm}

Since $E$ admits a complex orientation, it is useful to choose one and then think of a $C^n$-structure as a ratio\footnote{This is because the Thom isomorphism for negative rank bundles is used.} of elements of $E_*[[x_1,..,x_n]]$ satisfying normalization, symmetry, and cocycle conditions. The key insight of Ando-French-Ganter (proved in a coordinate-free language) is that a $C^n$-structure $f(x_1,...,x_n)$ defines a $C^{n-1}$-structure in the first $n-1$-variables, as long as one is willing to change coefficient rings from $E_*$ to $E_*((x_n))$. Moreover, the latter is isomorphic to the coefficient ring of $E^{t\T}$, which is again even periodic.

\begin{defn}\label{defn_afg_sharp} (\cite[Proposition 5.6]{afg})
    Let $n\in \{0,1,2,3\}$ and suppose $E$ is an even periodic homotopy ring spectrum with a homotopy $\MU\l 2(n+1)\r$-orientation $\omega$ with corresponding $C^{n+1}$-structure $f_\omega(x_1,...,x_{n+1})$. Define the \emph{AFG-sharp} orientation 
    \[(\omega)^{\mathrm{afg}\#}:\MU\l2n\r\to E^{t\T}\]
    to be the one corresponding (under \Cref{AHS_Cp_to_or}) to the $C^n$-structure on $E^{t\T}$ $f_\omega(x_1,,,,x_{n},t)$ as outlined in the previous paragraph.
\end{defn}

Our goal in the following section is to provide an $\EE_\infty$-refinement of this construction, which will turn out to apply to arbitrary values of $n$ and arbitrary commutative ring spectra (as opposed to even periodic ones).

\subsection{The sharp construction is $\EE_\infty$}

The orientation theory presented in \Cref{eul_or_cats} was centered around the cofiber sequence $\bu\to \ku\to \pi_0\ku$. Instead of $\bu$, we can use a general connective spectrum with a map to $\ku$.

\begin{defn}
Let $\xi$ be a connective spectrum with a map $\alpha \colon \xi \to \ku$ and let $\cC\in \calg(\cat^\perf)$. We denote by
 $\Or(\cC;\alpha)$ the space of null-homotopies of the composite 
\[
\xi \oto{\alpha} \ku \oto{j} \pic(\Sph) \to \pic(\cC). 
\]
Similarly, for $\phi \colon \Perf(\Sph)^{B\T} \to \cC$ we denote by $\Or^\proj(\cC,\phi;\alpha)$ the space of null-homotopies of the composite 
\[
\xi \oto{\alpha} \ku \oto{\tilde{j}} \pic(\Sph)^{B\T} \oto{\phi} \pic(\cC). 
\]
\end{defn}
Both constructions are covariantly functorial in $\cC$ and contravariantly functorial in $\alpha$. In particular, for every $\alpha \colon \xi \to \ku$ we have a canonical projective orientation $\ortate \circ \alpha$ of $\Perf(\Sph)^{t\T}$; here, we use the fact that $\ortate$ extends to a null-homotopy of the map $T \circ \tilde{j}$ from $\ku$, not only its restriction to $\bu$. Moreover, the construction $\omega \mapsto \omega^\proj$ from \Cref{defn_proj_or} immediately generalizes to a map 
\[
\Or(\cC;\alpha) \to \Or^\proj(\cC^{B\T},u;\alpha)
\]
for general $\alpha$. We get the following generalization of \Cref{defn_tch}.
\begin{defn}\label{defn_general_tch}
    For $\omega \in \Or(\cC;\alpha)$ define the Euler-Tate characteristic class 
\[
\tch:=\frac{T(\omega^\proj)}{u^{t\T} (\ortate \circ \alpha)} \colon \xi \to \glone(\one_{\cC^{t\T}}).
\]
\end{defn}

One advantage of this generalized perspective is that we can relate orientations with respect to different spectra $\xi$. 
%Namely,
%given a map $\alpha\colon \xi \to \ku$, we can obtain ``weaker'' versions of it as follows.

\begin{defn}
Let $\alpha \colon \xi \to \ku$ be a map of connective spectra.
Define a new map $\alpha^\flat$ as follows: 
\[
\alpha^\flat \colon \Sigma^2\xi \oto{\Sigma^2 \alpha} \Sigma^2 \ku \simeq \bu  \oto{i} \ku.
\]
\end{defn}
For example, if $i_n\colon \Sigma^{2n}\ku \to \ku$ is the $2n$-th connective cover map, then 
\[
i_n^\flat \simeq i_{n+1}.
\]
Our analog of the AFG-sharp construction will turn orientations of $\cC$ with respect to $\alpha^\flat$ into orientations of $\cC^{t\T}$ with respect to $\alpha$, under appropriate assumptions on $\alpha$. Namely, we will achieve this for maps $\alpha$ that are \emph{$\ku$-module maps}. 
Our main tools are the Euler-Tate orientation and the following map, which ``goes up the Whitehead tower.''

\begin{defn}\label{shear}
The class $[\CC]-[L]\in \bu^0B\T$ corresponds to a map 
\[
\Sigma^\infty_+ B\T \to \bu
\]
which linearizes over $\ku$ to a map 
\[
B\T \otimes \ku \to \bu.
\]
Denote the adjoint of this map by 
\[
\shear\colon  \ku \to \bu^{B\T}.
\]
\end{defn}

\begin{lem}
The composite
\[
\ku \oto{\shear} \bu^{B\T} \oto{(j|_{\bu})^{B\T}}  \pic(\Sph)^{B\T}\\         
\]
is canonically homotopic to the difference of maps 
\[
\pi^*\circ j  - \tilde{j} \colon \ku \to \pic(\Sph)^{B\T}
\]
where
$\pi^*\colon \Perf(\Sph) \to \Perf(\Sph)^{B\T}$ is the constant action functor.  
\end{lem}

\begin{proof}
We are considering the composite
\[
\xymatrix{
\ku \ar@/^1.5 pc/[rr]^{[\CC]-[L]} \ar^{\shear}[r] & \bu^{B\T} \ar[r] & \ku^{B\T} \ar^{j^{B\T}}[r] & \pic(\Sph)^{B\T}  
}
.\]
It will thus suffice to show that the composite
\[
\xymatrix{
\ku \ar^{[\CC]}[r] & \ku^{B\T} \ar^{j^{B\T}}[r] & \pic(\Sph)^{B\T}
}
\]
is homotopic to $\pi^*\circ j$ 
and that the composite
\[
\xymatrix{
\ku  \ar^{[L]}[r] & \ku^{B\T} \ar^{j^{B\T}}[r] & \pic(\Sph)^{B\T}
}
\]
is homotopic to $\tilde{j}$. The first claim follows from the fact that $\pi^* \colon \ku \to \ku^{B\T}$ is the unit of the $\ku$-algebra $\ku^{B\T}$ and the commutativity of the square 
\[
\xymatrix{
\ku  \ar^{\pi^*}[r]\ar[d]      & \ku^{B\T}  \ar[d]\\
\pic(\Sph) \ar^{\pi^*}[r] & \pic(\Sph)^{B\T}
}
.\]
The second claim follows from the fact that the map $[L]\colon \ku \to \ku^{B\T}$ is the mate of the action map $B\T \otimes \ku \to \ku$ which appears in the definition of $\tilde{j}$ (\Cref{defn_equi_j}).
%\[
%B\T \times \J \to \J.
%\]
\end{proof}

Combining this with the Euler-Tate orientation, which is a nullhomotopy of the composite of $\tilde{j}$ and the Tate quotient functor, we obtain the following:
\begin{cor} \label{cor:shear_eul_tate_j}
The Euler-Tate orientation induces a homotopy rendering the following diagram commutative:
\[
\xymatrix{
\ku \ar^{\shear}[r] \ar^{j}[d] & \bu^{B\T} \ar^{j^{B\T}}[r] & \pic(\Sph)^{B\T} \ar^{T}[d]\\ 
\pic(\Sph) \ar^{\pi^*}[r] &  \pic(\Sph)^{B\T} \ar^{T}[r]            & \pic(\Perf(\Sph)^{t\T})
}.
\]
\end{cor}

We shall now exploit this observation to show that taking Tate fixed points with respect to $\T$ allows one to ``improve'' complex orientations. 
Let $\xi$ be a $\ku$-module. We abuse notation and denote by $i\colon \Sigma^2\xi \to \xi$ the tensor product of $\xi$ with the map $i\colon \bu \to \ku$, and by $\shear\colon \xi \to (\Sigma^2\xi)^{B\T}$ mate of the map $B\T \otimes \xi \to \Sigma^2\xi$ obtained from the map $B\T\otimes \ku \to \bu$ of \Cref{shear} by tensoring with $\xi$ over $\ku$.

\begin{construction}\label{constr_sharp}
For a $\ku$-module map $\alpha \colon \xi \to \ku$ and $\cC\in \calg(\cat^\perf)$ we construct a map $(-)^{\sharp} \colon \Or(\cC;\alpha^\flat) \to \Or(\cC^{t\T};\alpha)$ as follows. 
By definition, $\Or(\cC;\alpha^\flat)$ is the space of null-homotopies of the composite 
\[
\Sigma^2\xi \oto{\alpha^\flat} \ku \to \pic(\Sph) \to \pic(\cC).
\]
Consider the diagram 
%\[
%\xymatrix{
%\xi \ar^{1_\xi \otimes_\ku \shear}[r] \ar@[red]^{\alpha}[d]        & \Sigma^2\xi^{B\T}\ar@[blue]^{(\alpha^\flat)^{B\T}}[rd] \ar^{1_\xi \otimes_\ku i}[r]\ar^{\Sigma^2\alpha^{B\T}}[d] & \xi^{B\T} \ar^{\alpha^{B\T}}[d]   &   \\ 
%\ku\ar@[red]^j[d]  \ar^{\shear}[r]       & \bu^{B\T} \ar^{i^{B\T}}[r] & \ku^{B\T}\ar@[blue]^{j^{B\T}}[d] & \\ 
%\pic(\Sph)\ar@[red]^u[d]\ar^{\pi^*}@{..>}[rr] &   &  \pic(\Sph)^{B\T} \ar@[blue]^{u^{B\T}}[d] \ar^{T}[r] & \pic(\Perf(\Sph)^{t\T}) \ar^{u^{t\T}}[d]  \\ 
%\pic(\cC) \ar@[red]^{\pi^*}[rr]          &   &  \pic(\cC)^{B\T} \ar@[red]^{T}[r]   & \pic(\cC^{t\T})    \\ 
%}.
%\]
\[
\begin{tikzcd}
    \xi \arrow[r,"\shear"] \arrow[d,red,"\alpha"]        & (\Sigma^2\xi)^{B\T}\arrow[rd,blue,"(\alpha^\flat)^{B\T}"] \arrow[r, "i^{B\T}"]\arrow[d,"\Sigma^2\alpha^{B\T}"'] & \xi^{B\T} \arrow[d,"\alpha^{B\T}"]   &   \\ 
\ku\arrow[d,red,"j"]  \arrow[r,"\shear"]       & \bu^{B\T} \arrow[r,"i^{B\T}"] & \ku^{B\T}\arrow[d,blue,"j^{B\T}"] & \\ 
\mathrm{pic}(\S)\arrow[d,red,"u"]\arrow[rr,dashed,"\pi^*"] &   &  \mathrm{pic}(\S)^{B\T} \arrow[d,blue,"u^{B\T}"] \arrow[r,"T"] & \mathrm{pic}(\Perf(\Sph)^{t\T}) \arrow[d,"u^{t\T}"]  \\ 
\mathrm{pic}(\cC) \arrow[rr,red,"\pi^*"]          &   &  \mathrm{pic}(\cC)^{B\T} \arrow[r,red,"T"]   & \mathrm{pic}(\cC^{t\T})    \\ 
\end{tikzcd}.
\]
 The upper squares commute by the naturality of the tensor product of $\ku$-modules, the lower right square again by naturality, and the entire hook-shaped bottom part commutes\footnote{Note that the two rectangles in the bottom left do not commute in general, as indicated by the dashed arrow which is part of both of them.} via \Cref{cor:shear_eul_tate_j}. 

By applying the functor $(-)^{B\T}$, the space $\Or(\cC,\alpha^\flat)$ maps to the space of null-homotopies of the composite $u^{B\T} \circ j^{B\T} \circ (\alpha
^\flat)^{B\T}$ (indicated in blue). Precomposing with $\shear \colon \xi \to (\Sigma^2\xi)^{B\T}$ and postcomposing with $T\colon \pic(\cC)^{B\T} \to \pic(\cC^{t\T})$, we land in the space of null-homotopies of the upper right composite in the diagram, which is equivalent (via the the homotopies rendering the diagram commutative) to the space of null-homotopies of the left-bottom composite (indicated in red). The latter space is, by definition, $\Or(\cC^{t\T},\alpha)$, so we have constructed a map
\[
(-)^\sharp \colon \Or(\cC,\alpha^\flat) \to \Or(\cC^{t\T},\alpha). 
\]
\end{construction}

Observe also that there is an evident map\footnote{This map is denoted $\omega \mapsto \delta(\omega)$ in \cite{afg}, in the special case that $\xi$ is an $n$-connected cover of $\ku$.} $i^*\colon \Or(\cC,\alpha) \to \Or(\cC,\alpha^\flat)$ given by precomposition with the map 
\[
i\colon \Sigma^2\xi  \to \xi.
\]
%\begin{defn}\label{def_tch_xi} {\red unfinished}
%    Let $\xi\to ku$ be the $n$-connected cover. We get an $n$-connected analog of the Tate-valued chern map (\Cref{defn_tch})
%    \[\tch : \xi\to \glone(R^{t\T}). \]
%\end{defn}
\begin{prop} \label{prop:sharp_restriction}
Let $\alpha \colon \xi \to \ku$ be a $\ku$-module map and let $\cC\in \calg(\cat^{\perf})$. For $\omega \in \Or(\cC,\alpha)$ let $\tch \colon \xi \to \glone(\one_{\cC^{t\T}})$ be the corresponding Euler-Tate characteristic class (\Cref{defn_general_tch}). There is a commutative diagram of spaces (natural in $\cC$): 
\[
\xymatrix{
\Or(\cC,\alpha)\ar^{i^*}[rr]\ar^{((\tch)^{-1},(T\pi^*)_*)}[d] & & \Or(\cC,\alpha^\flat) \ar^{(-)^\sharp}[d] \\
\Map(\xi,\glone(\one_{\cC^{t\T}}))\times \Or(\cC^{t\T},\alpha)   \ar^-{\mathrm{act.}}[rr] & & \Or(\cC^{t\T},\alpha),  
}
\]
where the bottom map is the free transitive action of $\Map(\Sigma^2\xi,\glone(\one_{\cC^{t\T}}))$ on the space of orientations.

Thus, for $\omega\in \Or(\cC,\alpha)$ we have
\[
(i^*\omega)^{\sharp} \simeq (\tch)^{-1}\cdot T\pi^*(\omega).  
\] 
\end{prop}

\begin{proof}
First, note that as a functor in $\cC\in \calg(\cat^\perf)$ the upper left corner of the diagram is co-represented by $\Perf(M\alpha)$, the perfect modules over the Thom spectrum of $\alpha$. Hence, to prove the commutativity naturally in $\cC$ it suffices to prove it for a single $\cC$ and after evaluation at a single point in the upper left corner (namely the universal $\alpha$-orientation of $M\alpha$). 
Both sides are paths in the space
$\Map(\Sigma^2\xi,\pic(\cC^{t\T}))$ connecting $0$ and the composite $T\pi^* u j  \alpha i$, with the notation as in the diagram of \Cref{constr_sharp}. 
To identify them more easily, we shall use the Eckmann-Hilton principle. 

Namely, given two maps $f,g\colon X\to Y$ of connective spectra and null-homotopies $h\colon 0 \iso f$ and $k\colon 0\iso  g$, by the Eckmann-Hilton argument the two paths 
\[
0 \oto{h} f \oto{k} f + g
\]
and 
\[
0 \oto{h + k} f + g 
\]
are homotopic. In particular, $-h\colon 0 \iso -f$ is homotopic to the $-f$-translate of $h^{-1}\colon f \to 0$. Now, recall that $\tch = \frac{T(\omega^\proj)}{u^{t\T}(\ortate)}$ (\Cref{defn_general_tch}). 
Note also that, in the notation of the current section, we have (cf. \Cref{defn_proj_or})
\[
\omega^\proj \simeq \omega^{B\T}[L]\pi^*
\]
where $[L]\colon \ku^{B\T} \to \ku^{B\T}$ is multiplication with the class of the tautological bundle, and that 
\[
\omega^{B\T}\pi^* \simeq \pi^*\omega.
\]
Thus, 
\begin{align*}
(\tch)^{-1} \cdot T\pi^* \omega &\simeq  -T\omega^{B\T}[L]\pi^* + u^{t\T} \ortate \alpha + T\omega^{B\T}\pi^* \simeq \\
&\simeq T\omega^{B\T}([\CC]- [L])\pi^* + u^{t\T}\ortate \simeq \\
&\simeq T\omega^{B\T}i^{B\T}\shear + u^{t\T} \ortate \alpha.
\end{align*}
where 
\begin{itemize}
\item the first identification follows by applying (twice) the Eckmann-Hilton argument as presented above,

\item the second identification is obtained by gathering the first and last terms, and 
\item the last identification is obtained from the equivalence $([\CC]-[L]) \pi^* \simeq i^{B\T}\shear$, which is immediate from the definition of $\shear$.
\end{itemize}

We now observe that the first summand $T\omega^{B\T}i^{B\T}\shear$ is precisely the homotopy contracting the composite $Tu^{B\T} j^{B\T} (\alpha^\flat)^{B\T} \tilde{\beta}$ in the diagram appearing in \Cref{constr_sharp} while the second summand $u^{t\T} \ortate \alpha$ is (up to translation) the homotopy rendering the bottom hook-shaped region commutative. Thus, the result follows from the construction of $\omega^{\sharp}$ in \Cref{constr_sharp}.  
\end{proof}

\begin{rem}\label{rem:swap_sharp_i}
In many cases, one can swap the order of $(-)^\sharp$ and $i^*$ in the result above. Namely, since $(-)^\sharp$ is natural in $\alpha$ and $i \circ \alpha \simeq \alpha^\flat \circ i$ for every $\ku$-module map $\alpha \colon \xi \to \ku$, we see that the square  
\[
\xymatrix{
\Or(\cC;\alpha^\flat) \ar^{\sharp}[r] \ar^{i^*}[d] & \Or(\cC^{t\T};\alpha) \ar^{i^*}[d] \\ 
\Or(\cC;(\alpha^\flat)^\flat) \ar^{\sharp}[r] & \Or(\cC^{t\T};\alpha^\flat)
}
\]
commutes naturally in $\cC\in \calg(\cat^{\perf})$. In particular, for $\omega \in \Or(\cC;\alpha^\flat)$ we deduce that 
\[
i^*(\omega^\sharp) \simeq (i^*\omega)^{\sharp} \simeq (\tch)^{-1}\cdot T\pi^*\omega. 
\]
\end{rem}

Next, in the special cases covered by \cite{afg}, we relate the $\sharp$-construction of Ando-French-Ganter to the $\EE_\infty$-multiplicative version of this section.  

\begin{thm}\label{thm_sharp_is_sharp}
    Let $E$ be an even periodic $\EE_\infty$-ring spectrum with an $\EE_\infty$-orientation $\omega: \MU\l 2(n+1)\r\to E$ for $n\in \{0,1,2,3\}$. Then the $\EE_\infty$ orientation $(\omega)^\sharp \colon \MU\l 2n\r\to E^{t\T}$ from \Cref{constr_sharp} is an $\EE_\infty$ lift of the AFG-sharp orientation $(\omega)^{\mathrm{afg}\#}$ of \Cref{defn_afg_sharp}.
\end{thm}

\begin{proof}
    By \Cref{AHS_Cp_to_or}, to identify the underlying homotopy ring map of $(\omega)^\sharp$ it suffices to identify the induced $C^n$-structure on $E^{t\T}$. This is the Thom class assigned by $(\omega)^\#$ to $\Lambda(n)$, which is an element $\theta_{\omega^\sharp}(\Lambda(n))\in (E^{t\T})^0(\mathrm{Th}(\Lambda(n)))$. 
    %{\red why its a class in the cohomology of that space? we divide by the ordinary orientation for that? as is it is a class in the cohomology of the Thom spectrum...}{\blue exactly, cf. the discussion just before \Cref{defn_afg_sharp}. Not that it's really necessary, we can write the Thom space here too }. 
    We will trace through the definition of $(\omega)^\sharp$ in steps. First, we form the $\ku\l 2(n+1)\r$-vector bundle with $\T$ action given by $\Lambda(n)_\triv-\Lambda(n)_\T$. Then we $E$-orient (produce an $E$-Thom class for) its image in $\pic(E)^{B\T}$ via the orientation $\omega^{B\T}$. This is equivalent to $E$-orienting the virtual bundle $\Lambda(n)\otimes 1-\Lambda(n)\otimes L$ over  $(\CP^\infty)^n\times B\T$ using the $\MU\l 2(n+1)\r$-orientation $\omega$. But that bundle is equivalent to $\Lambda(n+1)$, whose Thom class is the $C^{n+1}$-structure associated to $\omega$. Finally, we project this Thom class (via the functor $T$) to $E^{t\T}$. This is precisely the $C^n$-structure from the AFG-sharp construction (\Cref{defn_afg_sharp}).     
\end{proof}

\begin{rem}
    Since our sharp construction does not require $E$ to be even-periodic nor does it require $n=0,1,2,3$, it can be considered as a generalization of the AFG-sharp construction (in the $\EE_\infty$ setting).
\end{rem}

\begin{rem}\label{rem:cor_Frob_is_adj}
In \cite[Definition 5.12]{afg}, given a complex orientation $\omega$ of an even periodic homotopy ring spectrum $E$, the authors define the \emph{adjoint orientation} of $\omega$ to be (in our notation) the orientation $(i^*\omega)^{\mathrm{afg}\sharp}$. If $\omega$ is an $\EE_\infty$-orientation, then so is $i^*\omega$ and we obtain an $\EE_\infty$-lift $(i^*\omega)^{\sharp}$ of the adjoint orientation. Moreover, by \Cref{prop:sharp_restriction}, this lift can be identified with the twist of the original orientation of $E$ (after mapping it to $E^{t\T}$) by the $\EE_\infty$ characteristic class $(\tch)^{-1}$.  
\end{rem}

\subsection{The Jacobi orientation is $\EE_\infty$}

With \Cref{thm_sharp_is_sharp} in hand, we can prove that the the Jacobi orientation admits an $\EE_\infty$-lift.

\begin{defn}\label{def_jacobi_two_var}(\cite[Section 7]{afg})
    Write $\sigma:\MU\l 6\r \to \mathrm{tmf}$ for the complex string orientation (cf. \cite{ahr}). The \emph{Jacobi orientation} is the homotopy $\MSU$-orientation
    \[
        (\sigma)^{\mathrm{afg}\#}: \MSU\to \mathrm{tmf}^{t\T}
    \]
    gotten by applying the AFG-sharp construction (\Cref{defn_afg_sharp}).

    %Write $W:\MU\to KU[[q]]$ for the Witten genus of Tate $K$-theory. It is determined by the formula
    %\[
        %c_1(1-L)=(1-L)\prod_{k\geq 1}\frac{(1-q^kL)(1-q^kL^{-1})}{(1-q^n)^2}\in KU[[q]]^*\CP^\infty.
    %\]
    %Its adjoint $W^\mathrm{ad}$ (in the sense of \Cref{adjoint_genus}, see also \Cref{rem:cor_Frob_is_adj}) is called the \emph{two variable elliptic genus}.
\end{defn}

\begin{cor}\label{cor_jacobi_is_comm}
    The Jacobi orientation admits a lift to an $\EE_\infty$ $\MSU$-orientation.
\end{cor}
\begin{proof}
     By \cite{ahr}, $\sigma$ admits a lift to a map of $\EE_\infty$-rings. Hence by 
     \Cref{thm_sharp_is_sharp} $(\sigma)^\#$ is an $\EE_\infty$-lift of $(\sigma)^{\mathrm{afg}\#}$. 
     %Similarly, it follows from \cite{ahr} {\red ideally I find a better reference} that the Witten genus $W$ of Tate $K$-theory admits an $\EE_\infty$ lift. Hence by \Cref{thm_sharp_is_sharp} and \Cref{adjoint_genus} its $W^\mathrm{ad}$ does too.
\end{proof}

    %Terminology Adjoint of Witten genus is two variable elliptic AFG Section 6.2. Witten genus = sigma orientation for Tate curve. (about a page into Section 7) Jacobi orientation (implicitly of a fixed elliptic spectrum) = adjoint of canonical $\MU\l 6\r$ orientation of that elliptic spectrum. (about a page into Section 7). In hindsight we can probably just talk about the string orientation of $tmf$, which was not available at the time of AFG. 

\subsection{Some examples of orientations}\label{section_surprise_or}
In this section we consider the following special case of \Cref{prop:sharp_restriction} and investigate some consequences thereof.
\begin{cor}\label{thm_frob_factor_sharp}
    Let $R$ be an $\EE_\infty$-ring spectrum endowed with an $\EE_\infty$-orientation $\omega \colon \MU\l n \r \to R$ for $n\ge 2$. The adjoint orientation (\Cref{rem:cor_Frob_is_adj}) 
    $\omega^\adjoint \colon \MU\l n \r \to E^{t\T}$ 
    factors, as an $\EE_\infty$-ring map, through the sharp orientation 
    \[
        (\omega)^\#:\MU\l n-2\r\to R^{t\T}.
    \]  
\end{cor}
\begin{proof}
    This follows immediately from the combination of \Cref{prop:sharp_restriction} and \Cref{rem:swap_sharp_i}.
\end{proof}

In fact, let us further specialize to the case that $R=\MU\l n \r$ and $\omega$ is the identity map, which is of course $\EE_\infty$. We find that $\MU\l 2n\r^{t\T}$ admits an $\EE_\infty$ $\MU\l 2n-2\r$-orientation. Further specializing to low values of $n$ we find that $\MSU^{t\T}$ admits an $\EE_\infty$ $\MU$-orientation, and that $\MU^{t\T}$ admits an $\EE_\infty$ $\MUP$ orientation. Using this, we can provide an answer to a question of Hahn-Yuan (cf. \cite[Question 2]{hahnyuan}). 

Consider the group algebra $\S[\BU]$. Let $\beta:\S^2\to\S[\BU]$ be the Bott class. Snaith (\cite{snaithbu}) proved that there is an isomorphism of homotopy ring spectra
\[
\MUP_\mathrm{snaith}:=\S[\BU][\beta^{-1}]\simeq \MUP.
\]
In \cite{hahnyuan}, it is shown that $\MUP_\mathrm{snaith}$ is \emph{not} equivalent to $\MUP$ as an $\EE_\infty$-ring (in fact, not even as an $\EE_5$-ring). Therefore the following result is somewhat surprising. 
\begin{thm}\label{thm_both_MUP}
    $\MU^{t\T}$ admits an $\EE_\infty$-ring map from \emph{both} $\EE_\infty$-structures on periodic complex bordism. That is, there are $\EE_\infty$-ring maps
    \[\MUP\to \MU^{t\T}\]
    \[\MUP_\mathrm{snaith}\to \MU^{t\T}.\]
\end{thm}
\begin{proof}
    The first map is the one mentioned above, gotten as a special case of \Cref{thm_frob_factor_sharp}. The second map is gotten as follows. The space of $\EE_\infty$-ring maps $\S[\BU][\beta^{-1}]\to R$ is equivalent to the space of $\EE_\infty$-ring maps $f:\S[\BU]\to R$ such that $f(\beta)$ is an invertible element in $\pi_*R$. On the other hand, the adjunction
    \[
        \S[\Omega^\infty(-)]:\mathrm{Sp}^\mathrm{cn}\adj \mathrm{CAlg(Sp)}:\glone
    \]
    exhibits the space $\EE_\infty$-ring maps $f:\S[\BU]\to R$ as being equivalent to the space of maps of connective spectra $f^T:\bu\to\glone R$. Moreover, since $\beta\in\pi_2\S[\BU]$ lifts to an unstable class $\beta\in \pi_2\BU=\pi_2\bu$, we see that $f(\beta)$ corresponds to $f^T(\beta)$ under the canonical isomorphism $\pi_2R=\pi_2\glone R$. Therefore, to specify an $\EE_\infty$-ring map $\MUP_\mathrm{snaith}\to R$ it suffices to specify a map of spectra $\bu\to\glone R$ which sends $\beta\in \pi_2\bu$ to a unit in $\pi_*R$. We are interested in the case $R=\MU^{t\T}$ and we have a candidate map, namely  (cf. \Cref{defn_tch})
    \[
        \tch:\bu\to\glone \MU^{t\T}.
    \]
    It remains to check that $\beta\in \pi_2\bu$ is sent to an invertible element of $\pi_*R$. Now $\beta$ factors as 
    \[
        \S^2=\CP^1\to \CP^\infty\oto{L-1}\bu
    \]
    and \Cref{prop_char_series_tch} gives us an explicit formula for the value of $\tch$ on $L-1$; it is the characteristic series $t^{-1}(x+_Ft)$, where $F$ is the universal formal group law on $\pi_*\MU$ . The restriction of the characteristic series to $\CP^1$ corresponds to taking reduction modulo $x^2$, and so the image of $\beta$ in $\pi_2\MU^{t\T}$ is the coefficient of $x$ in $t^{-1}(x+_Ft)$. Write $F(y,z)=\sum a_{ij}y^iz^j$. We get
    \[
        \tch_*(\beta)=t^{-1}\sum_{j=0}^\infty a_{1j}t^j=t^{-1}(1+a_{11}t+a_{12}t^2+...)\in\pi_2\MU^{t\T}
    \]
    which is the product of two units in $\pi_*\MU^{t\T}=\pi_*\MU((t))$, and hence a unit.
\end{proof}

\begin{rem}
    Note that neither map defined above is an isomorphism, as $\pi_*\MUP$ is not isomorphic to $\pi_*\MU^{t\T}$; the former is countable in each degree while the latter is not.
\end{rem}

%\begin{rem}
%    {\red maybe too philosophical, can omit :)} It is interesting to consider where these two maps came from: the $\EE_\infty$-ring map $\MUP\to \MU^{t\T}$ is a consequence of the sharp construction  (\Cref{constr_sharp}), which at its core relies on the fact that the Tate construction vanishes on free orbits. The $\EE_\infty$-ring map $MUP_\mathrm{snaith}\to \MU^{t\T}$ is a consequence of the further fact that in the Tate construction, $t$ is both invertible \emph{and} a power series element, so that there are a lot of units in $\pi_*\MU^{t\T}$.
%\end{rem}

\Cref{thm_both_MUP} lets us also provide an answer to \cite[Question 5]{hahnyuan}, which asks about $\EE_\infty$ orientations of `forms of periodic integral homology' by `forms of periodic complex bordism'. In this language (cf. \cite[Remark 1.14]{hahnyuan}), $\Z^{t\T}$ is a form of periodic integral homology, and $\MUP$ and $\MUP_\mathrm{snaith}$ are forms of periodic complex bordism. We have the following result.

\begin{cor}\label{cor_forms_of_periodic}
    There are $\EE_\infty$-ring maps
    \[\MUP\to \Z^{t\T}\]
    \[\MUP_\mathrm{snaith}\to \Z^{t\T}.\]
\end{cor}
\begin{proof}
    We can compose each of the maps of \Cref{thm_both_MUP} with any $\EE_\infty$-ring map $\MU^{t\T}\to\Z^{t\T},$
    for example the one gotten by applying $(-)^{t\T}$ to the truncation map $\MU\to \Z$.
\end{proof}

\section{Further directions}
In this section we present some further directions for study.

There are real analogs of almost all of the results presented here, replacing $\C$ everywhere with $\R$. So for example, one uses the real $j$-homomorphism $\ko\to \pic(\S)$ instead of the complex one, the group $C_2\subset \R^\times$ instead of $\T$, and $\EE_\infty$ maps $\MO\to R$ instead of $\MU\to R$. The Euler-Tate class of \Cref{defn_tch} becomes a map of spectra
\[
w^\mathfrak{et}:\bo\to\glone R^{tC_2}
\]
The characteristic series is of the same form as in \Cref{prop_char_series_tch}, but the parameters are now of degree $1$ and correspond to real orientations\footnote{In particular, the formal group law is always isomorphic, but not necessarily equal, to the additive one.}. When $R=\F_2$, this produces an $\EE_\infty$-lift of the total Stiefel-Whitney class. 

There is another, simpler way to produce a lift of the Total Stiefel-Whitney class. If $u:R\to R^{tC_2}$ is the unit and $\mathrm{Frob}:R\to R^{tC_2}$ is the Tate-valued Frobenius of \cite{ns}, then one obtains an $\EE_\infty$-characteristic class valued in $\F_2^{tC_2}$ by applying the two natural maps $\F_2 \to \F_2^{tC_2}$ to its unique real orientation:  
\[
\omega = \frac{\mathrm{Frob}\circ\gamma}{u\circ \gamma},
\]
where $\gamma \colon \MO \to \F_2$ is the unique real orientation of $\F_2$. We expect that these constructions agree, so that  
\[
w^\mathfrak{et}\cdot (u\circ\gamma) = \mathrm{Frob}\circ\gamma
\] 
as $\EE_\infty$ maps $\MO\to R^{tC_2}$. This is a reflection of the classical fact that the total Stiefel-Whitney class can be recovered from the total Steenrod square of the Thom class. The absence of a ``Frobenius'' $R\to R^{t\T}$ makes the complex situation much more interesting, which is why we chose to focus our attention there.

The existence of a parallel story for $\R$ and $\C$ leads naturally to the following question.

\begin{ques}
    Is there a ``real'' variant of the projective orientation theory given here, which traces the complex conjugation action of $C_2$ on $\C$ and thus combines the real and complex projective orientation theories into one?
\end{ques}

Vector bundles, $j$-homomorphisms, Chern classes, and spectra all have motivic analogs, so it is natural to ask about motivic versions of projective orientation theory, which would formally imply the real version above by taking the base-field to be $\RR$.

\begin{ques}
    Are there motivic analogs of the projective orientation theory, and do they provide $\EE_\infty$-lifts of the motivic chern classes?
\end{ques}
In the case of the the total Chern class in $\ell$-adic cohomology, Akhil Mathew has pointed out that an $\EE_\infty$-lift might formally follow from our results if, after suitable $\ell$-completion, the Euler-Tate class
\[\tch:\bu_\ell\to\glone R^{t\T_\ell} \]
can be made equivariant for the $\Z_\ell^\times$ action by Adams operations on the source and by group automorphisms of $\T_\ell=B\Z_\ell$ on the target.

\bibliographystyle{plain}

\bibliography{references}

\begin{thebibliography}{10}

\bibitem{andocoherence}
M.~Ando.
\newblock Isogenies of formal group laws and power operations in the cohomology theories $e_n$.
\newblock {\em Duke Math. J.}, 79, 1995.

\bibitem{abghr}
M.~Ando, A.~Blumberg, D.~Gepner, M.~Hopkins, and C.~Rezk.
\newblock Units of ring spectra, orientations, and thom spectra via rigid infinite loop space theory.
\newblock {\em J. Topol.}, 2014.

\bibitem{afg}
M.~Ando, C.~French, and N.~Ganter.
\newblock The jacobi orientation and the two-variable elliptic genus.
\newblock {\em Alg. Geom. Topol.}, 8, 2008.

\bibitem{ahr}
M.~Ando, M.~Hopkins, and C.~Rezk.
\newblock Multiplicative orientations of $ko$-theory and the spectrum of topological modular forms.
\newblock 2010.

\bibitem{ahs}
M.~Ando, M.~Hopkins, and N.~Strickland.
\newblock Elliptic spectra, the witten genus, and the theorem of the cube.
\newblock {\em Invent. Math.}, 146, 2001.

\bibitem{barthelthom}
T.~Barthel and O.~Antolin-Camarena.
\newblock A simple universal property of thom spectra.
\newblock {\em J. Topol.}, 2018.

\bibitem{2varlib}
L.~Borisov and A.~Libgober.
\newblock Elliptic genera of singular varieties.
\newblock {\em Duke Math. J.}, 116, 2003.

\bibitem{lawsonetal}
C.~Boyer, B.~Lawson, P.~Lima-Filho, B.~Mann, and M.~Michelsohn.
\newblock Algebraic cycles and infinite loop spaces.
\newblock {\em Invent. Math.}, 113, 1993.

\bibitem{carmeliramzi}
S.~Carmeli, B.~Cnossen, M.~Ramzi, and L.~Yanovski.
\newblock Characters and transfer maps via categorified traces.
\newblock 2022.

\bibitem{2vardijk}
R.~Dijkgraaf, G.~Moore, and E.~Verlinde.
\newblock Elliptic genera of symmetric products and second quantized strings.
\newblock {\em Comm. Math. Phys.}, 185, 1997.

\bibitem{2vareguchi}
H.~Eguchi, H.~Ooguri, A.~Taormina, and S.~Yang.
\newblock Superconformal algebras and string compactification on manifolds with su(n) holonomy.
\newblock {\em Nucl Phys B}, 315, 1989.

\bibitem{hahnyuan}
J.~Hahn and A.~Yuan.
\newblock Exotic multiplications on periodic complex bordism.
\newblock {\em J. Topol.}, 13, 2020.

\bibitem{2varhirzell}
F.~Hirzebruch and R.~Jung.
\newblock Manifolds and modular forms.
\newblock {\em Aspects of Mathematics E20}, 1992.

\bibitem{hoplaw}
M.~Hopkins and T.~Lawson.
\newblock Strictly commutative complex orientation theory.
\newblock {\em Math. Z.}, 290, 2018.

\bibitem{2varkrichever}
I.~Krichever.
\newblock Generalized elliptic genera and baker-akhiezer functions.
\newblock {\em Mat. Zametki}, 47, 1990.

\bibitem{ns}
T.~Nikolaus and P.~Scholze.
\newblock On topological cyclic homology.
\newblock {\em Acta Math.}, 221, 2017.

\bibitem{segal}
G.~Segal.
\newblock The multiplicative group of classical cohomology.
\newblock {\em Q. J. Math.}, 26, 1975.

\bibitem{snaithbu}
V.~Snaith.
\newblock Localized stable homotopy of some classifying spaces.
\newblock {\em Math. Proc. Camb. Phil. Soc.}, 89, 1981.

\bibitem{2varwitten}
E.~Witten.
\newblock Elliptic genera and quantum field theory.
\newblock {\em Comm. Math. Phys.}, 109, 1987.

\end{thebibliography}

\end{document}